\documentclass{article}
\usepackage{amscd,amsfonts,euscript,amsmath,amsxtra,amssymb,amsthm}
\usepackage[english]{babel}
\newtheorem{theorem}{Theorem}
\newtheorem{lemma}{Lemma}

\newtheorem{proposition}{Proposition}

\theoremstyle{definition} \newtheorem{example}{Example}

\theoremstyle{remark}

\textwidth13.5cm \numberwithin{equation}{section}

\usepackage[all]{xy}
\def\A{{{\mathbb A}}}

\def\Q{{{\mathbb Q}}}

\def\P{{{\mathbb P }}}

\def\LL{{{\mathcal L}}}

\def\OO{{{\mathcal O}}}

\def\ZZ{{{\mathcal Z}}}

\def\Hom{{{\rm Hom }}}

\def\Spec{{{\rm Spec \,}}}

\def\Hilb{{{\rm Hilb \,}}}
\def\HHilb{{{\mathfrak{Hilb} \,}}}

\def\Supp{{{\rm Supp \,}}}

\def\Nil{{{\rm Nil \,}}}
\def\dim{{{\rm dim \,}}}

\def\ker{{{\rm ker \,}}}

\def\red{{{\rm red }}}

\begin{document}
{\sl MSC 14C05, 14B10}

 {\sl UDC 512.7}
\medskip

\begin{center}
{\Large\sc Elementary proof of reducedness of Hilbert schemes of
points in higher dimensions }\end{center}
\medskip
\begin{center}
Nadezda V. TIMOFEEVA

\smallskip

Yaroslavl' State University

Sovetskaya str. 14, 150000 Yaroslavl', Russia

e-mail: {\it ntimofeeva@list.ru}
\end{center}
\bigskip

\begin{quote}
The criterion for an affine primary algebra over the field to be
integral, is proven. Using this criterion we give a simple proof
that Hilbert scheme of 0-dimensional subschemes of length  $l$ of
nonsingular $d$-dimensional algebraic variety is reduced for all
$d$ and $l$.

Bibliography: 10 items.
\end{quote}


\markboth{Н.\,В.~Тимофеева}{Reducedness of Hilbert schemes of
points in higher dimensions}

\section{Introduction}
\label{subsec1} In the present article we work with algebraic
schemes of finite type over a field $k$ of zero characteristic.
This field is assumed to be algebraically closed. A variety is a
reduced Noetherian separated scheme of finite type over an
algebraically closed field. A curve is a variety understood as
described if it has dimension 1; a surface is a variety of
dimension 2. Hilbert scheme $\Hilb P$ whose points correspond to
closed subschemes of the scheme $P$ is a convenient and oftenly
used tool in algebro-geometric constructions as a simplest version
of (fine) moduli space. Following \cite{Hart1} we recall its
definition.

Let $f: X\to S$ be a morphism of schemes. For any $S$-scheme $S'
\to S$ set
\begin{equation*}
\HHilb_{X/S}(S')=\{S'\mbox{\rm -flat closed subschemes } \ZZ
\subset X\times_S S' \}
\end{equation*}
If $S''$ is any other  $S$-scheme and $S'' \to S'$ is $S$-morphism
than the map $\HHilb_{X/S}(S') \to \HHilb_{X/S}(S'')$ is defined
by the correspondence  $\ZZ \mapsto \ZZ \times _{S'} S''$. By this
means the contravariant functor $\HHilb_{X/S}$ is defined from the
category of schemes to the category of sets.

Let  $f:X\to S$ be a projective morphism, $\LL$ invertible
$\OO_X$-sheaf which is ample relative to $S$, $S$ is locally
Noetherian,  and $p(t)\in \Q[t]$ is a polynomial with rational
coefficients. Define $\HHilb^{p(t)}_{X/S}(S')$ as a set of
$S'$-flat subschemes $\ZZ$, such that the fibre  $Z_s=\ZZ \times
_{S'}s$ over each closed point  $s\in S'$ has Hilbert polynomial
equal to $p(t)$, i.e. $\chi((\OO_{\ZZ} \otimes
\LL^t)|_{f^{-1}(t)})=p(t)$. The symbol $\chi(\cdot)$ denotes the
(sheaf) Euler -- Poincar\`{e} characteristic. This defines a
subfunctor $\HHilb^{p(t)}_{X/S} \subset \HHilb_{X/S}$ and a
decomposition $\HHilb_{X/S}=\coprod _{p(t) \in \Q [t]}
\HHilb^{p(t)}_{X/S}.$ Now cite the classical A.~Grothendieck's
theorem.
\begin{theorem} \cite[theorem 3.2]{FGA} Let $X \to S$ be a
projective morphism of schemes and  $S$ be Noetherian. Then for
any polynomial $p(t) \in \Q[t]$ the functor $\HHilb^{p(t)}_{X/S}$
is representable by a scheme $\Hilb^{p(t)}(X/S)$ which is
projective over $S$. Hence the functor $\HHilb_{X/S}$ is also
representable, and its representing scheme is a disjoint union of
schemes  $\Hilb^{p(t)}(X/S).$
\end{theorem}

The scheme  $\Hilb^{p(t)}(X/S)$ is referred to as the {\it Hilbert
scheme of the scheme  $X$ over $S$ with Hilbert polynomial
$p(t)$.}



For our purposes we restrict ourselves by an "absolute" case when
 $S=\Spec k$.

In general  $\Hilb^{p(t)}P$ can have rather complicated structure.
In particular it can be nonreduced and can consist of several
connected components. This depends of the polynomial  $p(t)$ and
of the structure of the scheme $P$. The simplest case is a
constant Hilbert polynomial  $p(t)=l=h^0(\OO_Z)$. The natural
number  $l$ is called usually a {\it length} of zero-dimensional
subscheme $Z$. If  $Z$ is such that  $\OO_Z \cong
\bigoplus_{i=1}^l k_{x_i}$ where  $k_{x_i}$ is a skyscraper sheaf
whose nonzero fibre in the point $x_i$ is isomorphic to its
residue field (and, by algebraic closedness of the filed $k$, to
the filed $k$ itself), then the length of the subscheme $Z$ equals
 the number of points. In general case the sheaf  $\OO_Z$ is
isomorphic to the direct sum $\bigoplus_{i=1}^s A_i$ of Artinian
local $k$-algebras $A_i$ and $l= \sum_{i=1}^s \dim_k A_i$.

In the case when $P$ is a surface schemes ${\Hilb^l P}$ have been
studied in details during  70ies -- 90ies of the past century
(J.~Brian\c{c}on, J.~Fogarty \cite{Fogart}, R.~Hartshorne,
A.~Iarrobino \cite{Iarrobino}). If $P$ is nonsingular surface then
 $\Hilb^l P$ is nonsingular projective variety of dimension $2l$ --
 cf. \cite{Fogart}. Also in \cite{Fogart} it is proven that if
 $P$ is 3-dimensional variety then the scheme  $\Hilb^4 P$ contains
 singular points. Consequently all schemes $\Hilb^l P$ for $l\ge 4$
are singular. Also the conjecture that  $\Hilb^n \P^N$ are reduced
and irreducible for all $n$ and $N$ is formulated. There is a
classical result of R.~Hartshorne: if $P$ is relative projective
space over the connected scheme $S$ then Hilbert schemes
$\Hilb^{p(t)} P$ are connected provided they are nonempty (cf.
\cite{Hart1}). This result was reproven in the different way in
1996 by K.~Pardue \cite{Pardue} for usual ("absolute") projective
space over a field and for a projective space in zero
characteristic in 2004 by I.~Peeva and M.~Stillman
\cite{PeeStill}.  Let  $P$ be a nonsingular irreducible variety of
dimension  $d$. By the best knowledge of the author it is not
known till now about reducedness of schemes  $\Hilb^l P$  for
$l\ge 4$, $d\ge 3$.

Take a closed point   of the scheme  $\Hilb^l P$. Let it
correspond to the closed subscheme $i:Z\hookrightarrow P$. Denote
by $I_Z$ the sheaf of ideals  $\ker (\OO_P \to i_{\ast} \OO_Z)$.
The Zariski tangent space to the  scheme $\Hilb^l P$ at the point
$Z$ is isomorphic to the  $k$-vector space  $\Hom_{\OO_Z} (I_Z,
\OO_P/I_Z)$:
\begin{equation} T_Z \Hilb^l P \cong \Hom_{\OO_Z} (I_Z, \OO_P/I_Z).
\label{tanspace}
\end{equation}
Its dimension equals $ld$ if the point $Z$ of the scheme $\Hilb^l
P$ is general enough and corresponds to a subscheme $Z$ which
consists of $l$ distinct reduced points. A closed point
$i:Z\hookrightarrow P$ of the scheme $\Hilb^l P$  is called
singular if $\dim T_z
> ld$.

Questions on the dimension of the tangent space and on reducedness
of the scheme are local. Then we can replace an arbitrary
nonsingular $d$-dimensional variety by local neighborhood of its
closed point. This replacement and the introduction of coordinates
(in usual sense) are based on the result we recall following
\cite{Mats}.

Let $(A, \mathfrak m, K)$ be a local ring with maximal ideal $\mathfrak m$ and
residue field $K=A/\mathfrak m$. 
We need particular case of the theorem \cite[Theorem 29.7]{Mats}.
\begin{theorem} A complete regular local ring of characteristic 0
is  formal power series a ring over a field.
\end{theorem}

Since the base field  $k$ is assumed to be algebraically closed
then it is isomorphic to the residue field  $K$.

The cited theorem yields that it is enough to prove reducedness of
the scheme  $\Hilb^l \Spec k[[x_1, \dots, x_d]]$ with number of
indeterminates  $d$ equal to the dimension of the variety $P$ of
interest, since by to this theorem local neighborhoods of all
simple points of  $d$-dimensional varieties over algebraically
closed field of characteristic 0 are isomorphic. For convenience
of computations we fix standard homomorphism of rings $k[x_1,
\dots , x_d]\hookrightarrow k[[x_1,\dots, x_d ]]$ (inclusion into
the localization) and consider the scheme $\Hilb^l \A^d$ for
$\A^d=\Spec k[x_1, \dots, x_d]$.

If $\Nil \OO_X$ is a nilradical of the structure sheaf of some
scheme $X$ then its support  $\Supp \Nil \OO_X$ is a closed
subscheme in  $X$. In particular this means that if we are given a
flat 1-parameter family of subschemes  $\ZZ \to \Spec D$ where $D$
is an integral $k$-algebra of Krull dimension equal to 1 and
$\Supp \Nil \OO_{\Hilb^l P}$ contains an image of open subset $U
\subset \Spec D$, then $\Supp \Nil \OO_{\Hilb^l P}$ contains an
image of the whole of the curve  $\Spec D$. This reasoning will be
applied several times and we replace it by the short formulation:
{\it reduced point has reduced generization.} Similarly, {\it
nonsingular point has nonsingular generization.}

The main result of the present note is contained in the following
theorem.
\begin{theorem} Hilbert scheme $\Hilb^l P$ of  $l$-point subschemes
of the nonsingular $d$-dimens\-ional variety $P$ is reduced. Since
it is separated and has a finite type, then it is singular
$ld$-dimensional variety.
\end{theorem}

The article is organized as follows. In sect.2 we prove
irreducibility of the scheme  $\Hilb ^l P$ if the variety $P$ is
irreducible. In sect.3 bases of tangent spaces at most special
points of Hilbert schemes $\Hilb^{d+1} \A^d$ are written down.
Also we describe the deformations of the most special subscheme
$Z$ which continue basis tangent directions. In sect.4 we prove
that the computations done imply reducedness of Hilbert schemes of
$d+1$-point subschemes of  $d$-dimensional variety $P$. At last,
in sect.5 is shown that the reasonings done imply reducedness of
schemes $\Hilb^l P$ for all values $l$ and $\dim P=d$.

\section{Irreducibility} In this section we prove irreducibility
of the scheme  $\Hilb^l P$  in the case when the variety $P$ is
irreducible. The connectivity of this scheme is proven in just
cited classical work by R.~Hartshorne but connectivity does not
implies irreducibility.

To prove irreducibility of the scheme of interest it is enough to
confirm ourselves that all subschemes of length $l$ supported at
one point belong to the closure of the open subset in $\Hilb^l P$
formed by unions of $l$ reduced  points.

Let $Z$ be a subscheme supported at one point $p$ on the variety
$P$; choose a local isomorphism  $\OO_{P,p}\cong k[[x_1, \dots,
x_d]]$ so that  $\Supp Z=\mathfrak m=(x_1, \dots , x_d)$.

We construct the sequence of 1-parameter flat families
$Z_0(\alpha), Z_1(\alpha), \dots, Z_q(\alpha)$, $\alpha \in \P^1$,
such that  $Z_0(\infty)=Z$ and for  $\alpha \ne \infty$
$Z_q(\alpha) = \{p_1(\alpha), \dots, p_l(\alpha) \}$ is a
collection of $d$ distinct points.

Let the subscheme  $Z$ be defined by the ideal $I \subset k[[x_1,
\dots , x_d]]$.
Fix a natural ordering of indeterminates  $x_1<x_2< \dots < x_d$,
relate a lexicographic ordering of monomials to it in the ring
$k[x_1, \dots, x_d]$. Take a reduced Gr\"{o}bner basis $f_1, f_2,
\dots, f_m$ in $I$ with respect to this ordering. Then according
to  \cite[ch. 3  \S 1, theorem 2, \S 2, theorem 3]{CLO'S}, one
(the last) of polynomials $f_1, f_2, \dots, f_m$ contains
dependence of $x_1$ only. Note that all transformations done for
computations of Gr\"{o}bner bases do not change scheme structure;
then in cited theorems of elimination theory from \cite{CLO'S} the
term "variety" can be replaced by the term "scheme". Since the
initial subscheme $Z$ has a support at the point  ${\mathfrak
m}=(x_1, \dots, x_d)$ and the subscheme defined by the ideal
$(f_m)\subset k[x_1]$ contains Zariski closure for the image of
the subscheme $Z \subset \Spec k[x_1, \dots, x_d]$ under
projection to the line  $\Spec k[x_1] \subset \Spec k[x_1, \dots,
x_d]$, then $f_m=x_1^{l'}$ for  appropriate $l'\le l$.

Now consider the family of ideals  $I(\alpha)=(f_1. \dots,
f_{m-1}, x_d- \alpha f_m)$ where  $\alpha \in \P^1$. For $\alpha
=\infty$ we obtain the ideal  $I$ of the initial subscheme $Z$.
Ideals corresponding to  $\alpha \ne \infty,$ are taken to the
ideal  $I(0)=(f_1,\dots ,f_{m-1}, x_d)$ by corresponding
automorphism of the ring  $ k[x_1, \dots, x_d]$ such that $x_i
\mapsto x_i$ for $1\le i\le d-1$, $x_d\mapsto x_d+\alpha f_m$.
From this we conclude that all ideals  $I(\alpha)$ for $\alpha \ne
\infty$ define subschemes of equal lengths; let it  equal  $l_0$.
The family $I(\alpha)$ defines the family of subschemes $Z(\alpha)
\subset \Spec k[x_1, \dots, x_d] \times \P^1$ and a morphism of
the open subset of the base  $ \mu:\P^1\setminus {\infty} \to
\Hilb^{l_0}\Spec k[x_1, \dots, x_d]$. Consider a standard
immersion of the affine space $\Spec k[x_1, \dots, x_d]$ as affine
coordinate chart to the projective space $\P^d$ with homogeneous
coordinates $(X_0:X_1: \dots: X_d)$ such that $X_0\ne 0,$
$x_i=X_i/X_0$, $i>0$. It induces the immersion of Hilbert schemes
$\Hilb^{l_0} \Spec k[x_1, \dots, x_d] \subset \Hilb ^{l_0} \P^d$
and the composite morphism $ \mu:\P^1\setminus {\infty} \to
\Hilb^{l_0}\Spec k[x_1, \dots, x_d] \subset \Hilb ^{l_0} \P^d$.
Then there is a closed subscheme $Z(\alpha) \subset \P^d \times
\P^1 \to \P^1$. Since all closed fibres of the family $Z(\alpha)$
in points  $\alpha \in \P^1 \setminus \infty$ have constant length
which equals  $l_0$, then the family $Z(\alpha)$ is flat over
$\P^1 \setminus \infty$. According to \cite[ch. III, Proposition
9.8]{Hart}, there exist a unique closed subscheme in  $\P^d \times
\P^1$ which is scheme-theoretic closure of the subscheme
$Z(\alpha)|_{\alpha \ne \infty}$ and it is flat over $\P^1$. But
 required scheme-theoretic closure is precisely $Z(\alpha)$.
This proves that the scheme  $Z(\alpha)$ is flat over $\P^1$ and
hence all its closed fibres have equal lengths $l=l_0.$

From this we conclude that the scheme $Z$ belongs to the closure
of the locally closed subset whose points correspond to subschemes
with one-point support on nonsingular hypersurfaces in $\Spec
k[x_1, \dots, x_d]$. Hypersurfaces are defined by equations of the
form  $x_d-\alpha f_m=0$ where $f_m$ is a polynomial in variable
$x_1$. Every such hypersurface is isomorphic to the hyperplane
$x_d=0$, and we come to the analogous problem in the space of
lower dimension for a subscheme defined by the ideal $I'=(f_1,
\dots, f_{m-1}) \subset \Spec k[x_1, \dots, x_{d-1}]$. Choosing
reduced Gr\"{o}bner basis in it and continuing the process we
conclude that the initial subscheme $Z \subset \Spec k[x_1, \dots,
x_d]$ belongs to the closure of locally closed subset formed by
subschemes each of which lies on a smooth curve  $C$ defined by
the ideal  $(g_1, \dots, g_s)$. Under an appropriate choice of
coordinate system on this curve subschemes of locally closed
subset can be defined by ideals of the form $(x_1^l)$. Obviously,
the point corresponding to the subscheme of the described form
belongs to the closure of the set of points corresponding to
reduced $l$-point schemes.

\section{Basis vectors of the tangent space and their
continuing homomorphisms} \label{subsec2} Since $P$ is a scheme of
finite type over a field  (nonsingular variety) then  $\Hilb^l P$
is also of finite type over the field. The question about presence
or absence of nilpotents in the structure sheaf of the Hilbert
scheme allows us to reduce the consideration of general  $P$ of
dimension $d$ to the consideration of $\A^d=\Spec k[x_1, \dots,
x_d]$ at the neighborhood of the point $\mathfrak m=(x_1, \dots
x_d)$. We consider the point in the Hilbert scheme corresponding
to the subscheme  $Z \subset \A^d$ defined by the ideal
\begin{eqnarray*} \!\!\!\!\!\!\!I&=&(x_1, \dots, x_d)^2\nonumber \\
\label{ideal} &=&(x_1^2, x_1x_2, \ldots, x_1x_d, x_2^2, x_2 x_3,
\ldots, x_2x_d, \ldots, x_i^2, \ldots, x_ix_d, \ldots,
x_d^2).\end{eqnarray*}

Obviously, its length equals $l=d+1$.

This allows to write down explicitly the basis vectors  $v_{ijm}:
\!I \!\to k[x_1, \dots, x_d]/I $ of Zariski tangent space:
\begin{equation*}\label{basis}
v_{ijm}: x_rx_s \mapsto \left\{\begin{array}{l} x_m \mbox{ \rm for } (r,s)=(i,j),\\
0 \mbox{ \rm otherwise }\end{array}\right. \quad i\le j, \quad
r\le s.
\end{equation*}
Totally there are $r=d^2(d+1)/2=d^2l/2$ of linearly independent
vectors. One-parameter families of $(d+1)$-subschemes with  germs
constituting a basis of Zariski tangent space in the point of
interest, can be chosen, for example, as follows (they are
subdivided into two groups by geometrical meaning):\\
Group 1: reattachment of points. \begin{eqnarray*} i=1,\dots,d &&
(x_i(x_i-\alpha), x_i x_j, j\ne i, x_t x_j, t\ne i)\\ i=1,\dots,
d,\; j\ne i &&(x_i(x_i-\alpha), (x_i-\alpha)x_j, x_s x_t, (s,t)\ne
(i,j), (s,t)\ne (j,i)).
\end{eqnarray*} Group 2: subschemes on quadrics.
\begin{eqnarray*}
i,j=1,\dots, d,\; s\ne i,\; s\ne j,
\quad \quad
(x_i x_j-\nu x_s, x_q x_t,
(q,t)\ne (i,j), (q,t)\ne (j,i)).\\
\end{eqnarray*}
\begin{example} Set $d=3,$ $l=4.$ The point of the interest in the scheme
$\Hilb^4 \A^3$ corresponds to the ideal ${\mathfrak m}^2=(x^2, xy,
y^2,yz, z^2,zx)$. The dimension of the tangent space to the scheme
$\Hilb^4 \A^3$ equals 18, as well as in the general point it
equals 12. We enumerate 1-parameter families of 4-subschemes such
that germs of these families in the point of interest constitute a
basis of the tangent space.\\
\!\!\begin{tabular}{ll}1.$(x(x-\alpha), xy,y^2,
yz,z^2,zx),$&2.$(x^2,
xy, y(y-\beta),yz, z^2,zx),$\\
3.$(x^2, xy, y^2,yz,z(z-\gamma),zx),$& 4.$(x(x-\alpha), xy,
y^2,yz,
z^2,z(x-\alpha)),$\\
5.$(x^2\!, x(y-\beta), y(y-\beta),yz,z^2\!,zx)\!,$&6.$(x^2\!, xy,
y^2\!,y(z-\gamma), z(z-\gamma),zx))\!,$\\
7.$(x(x-\alpha), (x-\alpha)y,y^2\!, yz,z^2\!,zx)\!,$&8.$(x^2\!,
xy, y(y\!-\!\beta)\!,(y\!-\!\beta)z, z^2\!,zx)\!,$\\
9.$(x^2\!, xy, y^2\!,yz, z(z-\gamma),(z-\gamma)x)\!,$&10.$(x^2-\mu
y, xy, y^2\!,yz, z^2\!,zx)\!,$\\
11.$(x^2, xy, y^2-\nu z,yz, z^2,zx), $&12.$(x^2, xy, y^2,yz,
z^2-\sigma
x,zx),$\\
13.$(x^2-\pi z, xy, y^2,yz, z^2,zx),$&14.$(x^2, xy, y^2-\rho x,yz,
z^2,zx),$\\
15.$(x^2, xy, y^2,yz, z^2-\tau y,zx),$&16.$(x^2, xy-\eta z,
y^2,yz,
z^2,zx),$\\
17.$(x^2, xy, y^2,yz-\xi x, z^2,zx),$&18.$(x^2, xy, y^2,yz,
z^2,zx-\zeta y).$
\end{tabular}
\end{example}

\section{Reducedness}
Let $A$ be a commutative ring, $\Nil A$ its nilradical, $A_{\red}$
the quotient ring  $A/\Nil A$ which is called the {\it reduction}
of $A$. If $I \subset A$ is an ideal then set by the definition
$I_{\red}:=I/(I \cap \Nil A)$. It is clear that this is an ideal
in $A_{\red}$. Now prove the following simple lemma.
\begin{lemma} $({\mathfrak m}^2)_{\red}=({\mathfrak m}_{\red})^2$
\end{lemma}
\begin{proof}
It is clear that  $({\mathfrak m}_{\red})^2 \subset ({\mathfrak
m}^2)_{\red}.$ Take an element  $x\in ({\mathfrak m}^2)_{\red}$
and choose any of preimages $\overline x \in A$ for $x$;
$\overline x \in {\mathfrak m}^2$. Then  $\overline x=\overline y
\overline z,$ $\overline y, \overline z \in {\mathfrak m}$.
Denoting images $\overline y$ and $\overline z$ in the reduction
by $y$ and $z$ correspondingly we obtain $x=yz$, where $y,z \in
\mathfrak m_{\red}$.
\end{proof}

According to the lemma we omit brackets in the notation of the
square of the maximal ideal of the reduction: $\mathfrak
m_{\red}^2:=(\mathfrak m_{\red})^2=({\mathfrak m}^2)_{\red}$.

Since we work in the category of schemes of finite type over a
field $k$ then for our purposes it is enough to consider
commutative associative algebras of finite type over the field $k$
instead of arbitrary associative commutative rings with unity.
Otherwise speaking, the class of rings of our interest are
quotient algebras of polynomial rings  $k[x_1, \dots, x_N]$ in
appropriate number of indeterminates over the field $k$. These
 $k$-algebras are called for brevity  {\it affine} algebras.

Let $A$ be an affine algebra, $\varphi: k[x_1, \dots,
x_N]\twoheadrightarrow A$ the corresponding ring epimorphism,
$I=\ker \varphi$ its kernel, $I= \bigcap_i {\mathfrak q}_i$ the
primary decomposition of the kernel. Here ${\mathfrak q}_i$ is a
primary ideal for any  $i$. Then the examination of the
reducedness of the scheme $\Spec A$ reduces to the consideration
of its Zariski irreducible components  $\Spec A_i$ for algebras of
the special form  $A_i=k[x_1, \dots, x_N]/{\mathfrak q}_i$. Such
$k$-algebras are called {\it primary}. It is clear that for
 $A=k[x_1, \dots, x_N]/{\mathfrak q}$ we have
$A_{\red}=k[x_1, \dots, x_N]/\sqrt{\mathfrak q}$.
\begin{proposition}\label{propdim} The primary $k$-algebra
$A$ is integral if and only if $$\dim_k {\mathfrak m}/{\mathfrak
m}^2=\dim_k {\mathfrak m}_{\red}/{\mathfrak m}^2_{\red}.$$
\end{proposition}
\begin{proof} Consider the exact diagram of  $A$-modules
\begin{equation*}\xymatrix{&0&0&0\\
0\ar[r]&\Nil A/(\mathfrak m^2 \cap \Nil A) \ar[u] \ar[r]&
{\mathfrak m}/{\mathfrak m}^2 \ar[u] \ar[r]& {\mathfrak
m}_{\red}/{\mathfrak m}^2_{\red} \ar[u] \ar[r]& 0\\
0\ar[r]& \Nil A \ar[u] \ar[r]& \mathfrak m \ar[u] \ar[r] &
\mathfrak m_{\red} \ar[u] \ar[r]& 0\\
0\ar[r]&\mathfrak m^2 \cap \Nil A \ar[u] \ar[r]& {\mathfrak m}^2
\ar[u] \ar[r]& {\mathfrak m}^2_{\red} \ar[u] \ar[r]& 0\\
& 0\ar[u]& 0\ar[u]&0\ar[u]}
\end{equation*}
The equality $\dim_k {\mathfrak m}/{\mathfrak m}^2=\dim_k
{\mathfrak m}_{\red}/{\mathfrak m}^2_{\red}$ in the formulation of
the proposition means that ${\mathfrak m}/{\mathfrak
m}^2={\mathfrak m}_{\red}/{\mathfrak m}^2_{\red}$. This implies
$\mathfrak m^2 \cap \Nil A=\Nil A$ and  $ \Nil A \subset \mathfrak
m^2$.

Let $\eta \in \Nil A$ be a nilpotent element of index  $\iota>1$.
Since $\Nil A \subset \mathfrak m^2$ then  $\eta=\eta_1 \xi_1$ for
certain  $\eta_1, \xi_1 \in \mathfrak m$. Since  $A$ is affine
algebra then choose any preimages $\overline \eta_1, \overline
\xi_1 \in k[x_1, \dots, x_N]$ for $\eta_1, \xi_1$ under the
epimorphism $k[x_1, \dots, x_N]\twoheadrightarrow A$ with kernel
$\mathfrak q \subset k[x_1, \dots, x_N]$ for a primary ideal
$\mathfrak q \subset k[x_1, \dots, x_N]$. Then  $\overline
\eta_1^{\iota} \overline \xi_1^{\iota} \in \mathfrak q $, and by
primarity of the ideal  $\mathfrak q$ we conclude that $\eta_1,
\xi_1$ are nilpotent in  $A$. If one of them belongs to $\mathfrak
m \setminus \mathfrak m^2$ then $\dim_k {\mathfrak m}/{\mathfrak
m}^2>\dim_k {\mathfrak m}_{\red}/{\mathfrak m}^2_{\red}$, and the
contradiction completes the proof. Let as before $\eta_1, \xi_1
\in \mathfrak m^2$, and we can apply the reasoning described
above, say, to $\eta_1=\eta_2 \xi_2$, etc. We come to strictly
ascending chain of principal ideals  $(\eta) \subset (\eta_1)
\subset (\eta_2) \subset \dots$. The ascending chain condition for
ideals in $k[x_1, \dots , x_N]$ yields that there is a nilpotent
belonging to $\mathfrak m \setminus \mathfrak m^2$, and then
$\dim_k {\mathfrak m}/{\mathfrak m}^2>\dim_k {\mathfrak
m}_{\red}/{\mathfrak m}^2_{\red}$, what contradicts the condition
of the proposition.

Hence $\Nil A=0$, then $\mathfrak q$ is prime, $\mathfrak q =\sqrt
\mathfrak q$ and $A \cong k[x_1, \dots, x_N]/\sqrt \mathfrak q$.
The opposite implication is obvious and the proof is complete.
\end{proof}

Since Zariski tangent space to the scheme $\Spec A$ in its closed
point  ${\mathfrak m}$ is given by the vector space $(\mathfrak
m/\mathfrak m^2)^{\vee},$ then the equality $\dim_k (\mathfrak
m/\mathfrak m^2)=r$ means existence of $r$ $k$-linearly
independent homomorphisms of $k$-algebras $v_i: A \to
k[\varepsilon]/(\varepsilon^2)$. Let the homomorphism $v:A \to
k[\varepsilon]/(\varepsilon^2)$ factors as  $v:A
\stackrel{\phi}{\to} D \to k[\varepsilon]/(\varepsilon^2)$ where
$D$ is integral $k$-algebra. For convenience of computations we
can assume integral domain $D$ to have Krull dimension equal to
1. Then  $\Nil A \subset \ker \phi$, and $\phi$ and $v$ factor
through $A_{\red}$. If all basis vectors $v_i, i=1, \dots , r$
have this property then we are in the realm of the proposition
\ref{propdim}, and $A$ is integral. The opposite is obviously
true.

Geometrically this means the following: there are $r$ curves
through the point $\mathfrak m$ of the scheme $\Spec A$ such that
their tangent vectors are linearly independent.

\begin{example} The requirement of primarity of the ring  $A$
is not superfluous. For example the scheme $\Spec k[x,y]/(x^2y)$
is nonreduced at the point $(x,y)$, but there is a basis of
tangent space which consists of 2 homomorphisms $v_i:
k[x,y]/(x^2y)\to k[\varepsilon]/(\varepsilon^2)$, $i=1,2$, defined
by correspondences $v_1: x\mapsto \varepsilon,$ $v_1:y\mapsto 0$
and $v_2: x\mapsto 0,$ $v_2: y\mapsto \varepsilon$. These
homomorphisms factor through obvious homomorphisms $\phi_1:
k[x,y]/(x^2y) \to k[x]$ and $\phi_2: k[x,y]/(x^2y) \to k[y]$
respectively.

Performing primary decomposition we come to two components
$A_1=k[x,y]/(y)=k[x]$ и $A_2=k[x,y]/(x^2)$ where only one is
reduced. The same conclusion is provided by our criterion.
\end{example}

\section{Proof for all schemes  $\Hilb^l \A^d$}

It is known that schemes  $\Hilb^l\A^{l-1}$ are reduced for all
 $l\ge 2$. Show that this implies that any scheme $\Hilb^l \A^d$
and hence any $\Hilb^l P$ is reduced where $P$ is $d$-dimensional
nonsingular variety.

If $d>l-1$ then reducedness of the scheme $\Hilb^l \A^d$ follows
from specializations. For the proof it is necessary to construct a
deformation of the most special point of the scheme $\Hilb^{d+1}
\A^d$ to disjoint union of an appropriate subscheme  $Z'\in
\Hilb^l \A^d$ supported at a point, and $d-l+1$ reduced points. It
is enough to consider the case when  $Z'$ is a subscheme of the
most special form. Under appropriate choice of coordinate system
it is defined by the ideal  $I'=( x_i x_j, 1\le i\le j \le l-1,
x_l, \dots, x_d)$. The desired deformation is defined by the ideal
$(x_i x_j, 1\le i\le j \le l-1, x_s (x_t-\alpha_t), 1\le s\le d,
l\le t \le d, s\le t).$ Since reduced point has reduced
generization then the proof for
 $d>l-1$ is complete.

Now let $d<l-1$. Fix the immersion of the plane $\A^d
\hookrightarrow \A^{l-1}$ by vanishing of $l-d-1$ last
coordinates: $x_{d+1}=\dots =x_{l-1}=0,$ and the induced immersion
of the Hilbert scheme $\Hilb^{l}\A^d \hookrightarrow \Hilb^l
\A^{l-1}.$ Since we prove reducedness it is enough to consider the
most special point of the scheme $\Hilb^l\A^d$.  Let the point $Z$
of the scheme  $\Hilb^l\A^d$ correspond to the ideal  $I$ with
 generators  $f_1, \dots , f_s$. Then the image of this
point in the scheme  $\Hilb^l \A^{l-1}$ corresponds to the
subscheme with the ideal $I'=(f_1, \dots ,f_s, x_{d+1},
\dots,x_{l-1})$. Since the scheme  $\Hilb^l \A^{l-1}$ is reduced
then there exists a collection of curves on the scheme  $\Hilb^l
\A^{l-1}$ through the point  $Z$ with tangent directions forming a
basis of tangent space $T_Z \Hilb^l \A^{l-1}$. These tangent
directions are enumerate by the correspondences
$$f_i \mapsto \overline f_j, \;\; \overline f_j\in k[x_1, \dots,
x_d]/I,\;\;\; f_r \mapsto 0,\;\; r\ne i,\;\;\; x_q \mapsto 0,
\;\;d+1 \le q \le l-1,$$ and
$$f_i \mapsto 0, \;\; 1\le i \le s,\;\;\; x_t \mapsto \overline
f_j,\;\;\; x_q\mapsto 0, \;\;q\ne t.$$

Correspondences of the second group generate 1-parameter families
of the form \linebreak $(f_1, \dots, f_s, x_{d+1}, \dots,
x_{t-1},x_t-\alpha \overline f_j, x_{t+1}, \dots, x_{l-1})$. Since
the scheme
 $\Hilb^l \A^{l-1}$ is reduced, there exist
$\dim T_Z \Hilb^l \A^{l-1}$ curves in it with linearly independent
tangent directions i.e. 1-parameter subfamilies of the family
\begin{eqnarray}
&&\!\!\!\!\!\!\!\!\!\!(f_1-\sum_j \alpha_{1j} \overline f_j,
\dots, f_u-\sum_j \alpha_{uj}\overline f_j, \dots, f_s-\sum_j
\alpha_{sj}\overline f_j, \nonumber \\ \label{fam1}
&&\!\!\!\!\!\!\!\!\!\!x_{d+1}-\sum_j \beta_{d+1, j} \overline f_j,
\dots , x_t-\sum_j\beta_{tj}\overline f_j, \dots, x_{l-1}-\sum_j
\beta_{l-1,j} \overline f_j).
\end{eqnarray}
In such expression of deformations some of $\alpha_{uj}$ and of
$\beta_{tj}$ can equal  0 in all families; it depends on the
structure of the ideal $I$. The automorphism of the ring $k[x_1,
\dots, x_{l-1}]$ defined by the correspondence  $x_i \mapsto x_i$
for $i \le d$, $x_t \mapsto x_t+ \sum_j\beta_{tj} \overline f_j$
for $d+1 \le t \le l-1$ takes this family to the  family
\begin{equation*}
(f_1-\sum_j \alpha_{1j} \overline f_j, \dots, f_u-\sum_j
\alpha_{uj}\overline f_j, \dots, f_s-\sum_j \alpha_{sj}\overline
f_j, x_{d+1}, \dots , x_t, \dots, x_{l-1}).
\end{equation*}
We come to the projection to  $\Hilb^l \A^d$ for all families
chosen. The number of independent tangent directions in the image
of the projection equals to the dimension  $\dim T_Z \Hilb^l
\A^d$. Considering $\dim T_Z \Hilb^l \A^{l-1}$ curves with
linearly independent tangent directions provided by the family
(\ref{fam1}),  and applying the projection described, we come to
$\dim T_Z \Hilb^l \A^d$ curves in $\Hilb ^l \A^d$ with linearly
independent directions. This proves reducedness of the scheme
$\Hilb^l \A^d$ in its point $Z$, and hence everywhere.



\end{document}